\documentclass[12pt,a4paper]{article}
\usepackage[utf8]{inputenc}
\usepackage{amsmath,amsthm,dsfont,pifont,amsfonts,MnSymbol}

\usepackage[hidelinks]{hyperref}

\title{Open images of spaces with a Lusin $\pi\nos$-base\footnote{2020 Mathematics Subject Classification\textup{:} Primary 54{E}99; Secondary 54{C}10. Keywords\textup{:} pi-space, Lusin pi-base, the Baire space, Souslin scheme, open map}
}
\author{Mikhail Patrakeev\footnote{Krasovskii Institute of Mathematics and Mechanics of UB RAS, 620108, Yekaterinburg, Russia; \textit{e-mail address}\textup{:} p17533@gmail.com}\;,
Vlad Smolin\footnote{Krasovskii Institute of Mathematics and Mechanics of UB RAS, 620108, Yekaterinburg, Russia and Ural Federal University, Mathematical Analysis Department, 620002 Ekaterinburg, Russia; \textit{e-mail address}\textup{:} SVRusl@yandex.ru}}
\date{September 2021}

\theoremstyle{plain}
\newtheorem{teo}{Theorem}
\newtheorem{lemm}[teo]{Lemma}

\newtheorem{prop}[teo]{Proposition}

\theoremstyle{definition}
\newtheorem{deff}[teo]{Definition}

\newtheorem{nota}[teo]{Notation}

\newtheorem{rema}[teo]{Remark}
\newtheorem{ques}[teo]{Question}

\newcommand{\nos}{\mathsurround=0pt}
\renewcommand{\ll}{\langle}
\newcommand{\rr}{\rangle}
\renewcommand{\l}{\lceil}
\renewcommand{\r}{\rceil}
\renewcommand{\leq}{\leqslant}
\renewcommand{\geq}{\geqslant}
\newcommand{\bset}{{}^\omega\hspace{-0.5pt}\omega}
\newcommand{\cset}{{}^\omega2}
\newcommand{\bspace}{\mathcal{N}}
\newcommand{\btree}{{}^{{<}\omega}\hspace{-0.5pt}\omega}
\newcommand{\ctree}{{}^{{<}\omega}2}
\newcommand{\cl}[2]{\mathsf{Cl}_{#2}(#1)}
\renewcommand{\int}[2]{\mathsf{Int}_{#2}(#1)}
\newcommand{\fruit}[2]{\mathsf{fruit}_{#2}(#1)}
\newcommand{\shoot}[2]{\mathsf{shoot}_{#2}(#1)}
\newcommand{\uph}{{\upharpoonright}\hspace{0.5pt}}
\newcommand{\taun}{\tau_{\scriptscriptstyle\mathcal{N}}}

\DeclareMathOperator{\lh}{\mathsf{length}}
\newcommand{\inn}{\to}
\newcommand{\ninn}{\not\to}

\begin{document}

\maketitle

\begin{abstract}
In \cite{PATRAKEEV2023108410} we studied spaces with a Lusin $\pi$-base and $\pi$-spaces and posed the following question: Does the class of continuous open images of spaces with a Lusin $\pi$-base equal the class of continuous open images of $\pi$-spaces? We give a negative answer to this question.
\end{abstract}

\section{Introduction}
In \cite{PATRAKEEV2023108410} we introduced the notion of $\pi$-spaces: these are topological spaces that can be mapped onto the Baire space  (i.e. the countable power of the countable discrete space) by a continuous quasi-open bijection. A topological space is a continuous open image of a $\pi$-space if and only if it is a Choquet space of countable $\pi$-weight and of cardinality not greater than continuum~\cite{PATRAKEEV2023108410}. A second-countable space is a continuous open image of a $\pi$-space if and only if it is a continuous open image of a space with a Lusin $\pi$-base (and if and only if it is a Choquet space of cardinality not
greater than continuum)~\cite{PATRAKEEV2023108410}. The last result motivates the following question: does the class of continuous open images of $\pi$-spaces equal the class of continuous open images of spaces with a Lusin $\pi$-base?

We give a negative answer to the above question by constructing a zero-dimensional $\pi$-space that is not a continuous open image of a space with a Lusin $\pi$-base, see Theorem \ref{teo.main}. To achieve this result we give a description of open images of spaces with a Lusin $\pi$-base: these are topological spaces of cardinality not grater than continuum and with an  $\alpha$-scheme, see Theorem~\ref{desc_open_LPB}. 
\bigskip

\section{Notation and terminology}

We use terminology from~\cite{topenc} and~\cite{kunen2014set}. A \emph{space} is a topological space. We also use the following notation.

\begin{nota}\label{not01}%
  The symbol $\coloneq$ means ``equals by definition''\textup{;}
  the symbol ${\colon}{\longleftrightarrow}$ is used to show that the expression on the left side is an abbreviation for the expression on the right side\textup{;}
  \begin{itemize}
  \item [\ding{46}\,]
 $\mathsurround=0pt
 \omega$ $\coloneq$ the set of finite ordinals $=$ the set of natural numbers, so $0=\varnothing\in\omega$ and ${n}=\{0,\ldots,{n}-1\}$ for all ${n}\in\omega;$
  \item [\ding{46}\,]
 $\mathsurround=0pt
 {s}\,$ is a \textit{sequence}$\ {\colon}{\longleftrightarrow}\
 {s}$ is a function whose domain is a finite ordinal or is $\omega$;
\item [\ding{46}\,]
 if ${s}$ is sequence, then

 $\mathsurround=0pt
 \lh({s})\:\coloneq$ the domain of ${s}$;
 \item [\ding{46}\,]
   $\mathsurround=0pt
   \langle{s}_0,\ldots,{s}_{{n}-1}\rangle$ $\coloneq$
   the sequence ${s}$ such that  $\lh({s})={n}\in\omega$ and  ${s}({i})={s}_{i}$ for all ${i}\in{n};$
 \item [\ding{46}\,]
   $\mathsurround=0pt
   \langle\rangle$ $\coloneq$ the sequence of length {0};
 \item [\ding{46}\,]
   if ${s}=\langle{s}_0,\ldots,{s}_{{n}-1}\rangle$, then

   $\mathsurround=0pt
   {s}\hspace{0.5pt}^{\frown}{x}
   \:\coloneq\:\langle{s}_0,\ldots,{s}_{{n}-1},{x}\rangle$;
\item [\ding{46}\,]
 $\mathsurround=0pt
 {f}\uph{A}$ $\coloneq$ the restriction of the function ${f}$ to the set ${A};$

 \item [\ding{46}\,] ${g}\circ {f}$ is the composition of functions ${g}$ and ${f}$ (that is, ${g}$ after ${f}$);
  \item [\ding{46}\,]
 $\mathsurround=0pt
 {A}\subset{B}
 \ {\colon}{\longleftrightarrow}\
 {A}\subseteq{B}\enskip\text{and}\enskip{A}\neq{B};$
  \item [\ding{46}\,]
 if ${s}$ and ${t}$ are sequences, then

 $\mathsurround=0pt
 {s}\sqsubseteq{t}
 \ {\colon}{\longleftrightarrow}\
 {s}={t}\uph\lh({s})\ $ and

 $\mathsurround=0pt
 {s}\sqsubset{t}
 \ {\colon}{\longleftrightarrow}$
 ${s} \sqsubseteq{t}$ and ${s}\neq{t}$

 (actually, ${s}\sqsubseteq{t}\leftrightarrow{s}\subseteq{t}$
 and  ${s}\sqsubset{t}\leftrightarrow{s}\subset{t}$);
  \item [\ding{46}\,]
 $\mathsurround=0pt
 {}^{B}\!{A}$ $\coloneq$ the set of functions from ${B}$ to ${A};$

 in particular, ${}^{0}\hspace{-1pt}{A}=\big\{\langle\rangle\big\};$
\item [\ding{46}\,]
 $\mathsurround=0pt
 {}^{{<}\hspace{0.5pt}\omega}\hspace{-1.5pt}{A}
 \:\coloneq\:\bigcup_{{n}\in\omega}{}^{{n}}\hspace{-1pt}{A}\,=\,$
 the set of finite sequences in ${A};$
  \item [\ding{46}\,]
 $\mathsurround=0pt
 [{A}]^\kappa$ $\coloneq$ the set of subsets of ${A}$ of cardinality $\kappa$;

\item [\ding{46}\,]
 if ${p}$ is a point in a space with topology $\tau$, then

 $\mathsurround=0pt
 \tau({p})$ $\coloneq$ $\{{U}\in\tau:{p}\in{U}\}$ $=$
 the set of open neighbourhoods of ${p}$;
  \item [\ding{46}\,]
 $\mathsurround=0pt
 \gamma\,$ is a $\pi\mathsurround=0pt$-\textit{net} for a space ${X}
 \ {\colon}{\longleftrightarrow}\ $ all elements of $\gamma$ are nonempty and for each nonempty open ${U}\subseteq{X}$, there is ${G}\in\gamma$ such that ${G}\subseteq{U}$;
  \item [\ding{46}\,]
 $\mathsurround=0pt
 \gamma\,$ is a $\pi\mathsurround=0pt$-\textit{base} for a space ${X}
 \ {\colon}{\longleftrightarrow}\ \gamma$ is a $\pi\mathsurround=0pt$-net for ${X}$ and all elements of $\gamma$ are open;
  \item [\ding{46}\,]
  $\mathsurround=0pt
  \taun$ $\coloneq$  the Tychonoff product topology on the set $\bset$, where $\omega$ carries the discrete topology;
  \item [\ding{46}\,]
  $\mathsurround=0pt
  \bspace$ $\coloneq$ the \emph{Baire space} $=$ the space $\langle\bset,\taun\rangle$;
  \item [\ding{46}\,]
  $\mathsurround=0pt
  \mathfrak{c}$ $\coloneq$ the cardinality of the continuum.
  \end{itemize}
\end{nota}

\begin{nota}
  Let $\ll {X}, \tau \rr$ be a space and ${A} \subseteq X$. Then

  \begin{itemize}
  \item [\ding{46}\,] $\cl{A}{\tau} \coloneq $ the closure of ${A}$ in $\ll {X}, \tau \rr$;
  \item [\ding{46}\,] $\int{A}{\tau} \coloneq $ the interior of ${A}$ in $\ll {X}, \tau \rr$.
  \end{itemize}
 If the topology is clear from a context, then we omit an index in the above notations $\mathsf{Cl}$ and $\mathsf{Int}$.
\end{nota}

Recall that, in \cite{kechris2012classical}, a \emph{Souslin scheme}  is an indexed family $\langle V_a \rangle_{a \in \btree}$ of sets.

\begin{deff}
 Let ${\bf V} = \langle V_a \rangle_{a \in \btree}$ be a Souslin scheme, $\ll{X},\tau\rr$ be a space, and $p \in \bset$. Then

 \begin{itemize}
 \item[\ding{46}\,] ${\bf {V}}$ has nonempty leaves $\ {\colon}{\longleftrightarrow}\ $ $V_{a} \neq \varnothing$ for all $a \in \btree$;
 \item [\ding{46}\,]  ${\bf V}$  {\it covers} ${X}\ {\colon}{\longleftrightarrow}\  V_{\langle \rangle} = X$ and $V_a = \bigcup_{n \in \omega} V_{a\hspace{0.5pt}^{\frown}n}$ for all $a \in \btree$;
 \item [\ding{46}\,]  ${\bf V}$  {\it partitions} ${X}\ {\colon}{\longleftrightarrow}\  {\bf V}$ covers ${X}$ and  $V_{a\hspace{0.5pt}^{\frown}{n}} \cap V_{a\hspace{0.5pt}^{\frown}{m}} = \emptyset$ for all ${a} \in \btree$ and ${n} \neq {m} \in \omega$;
 \item [\ding{46}\,] $\mathsf{flesh}({\bf {V}})\coloneq\bigcup_{{a} \in \btree} {V}_{a}$;
 \item [\ding{46}\,] $\fruit{p}{{\bf V}}\coloneq\bigcap_{n \in \omega}V_{p \upharpoonright n}$;
 \item [\ding{46}\,]  ${\bf V}$ is {\it complete} $\ {\colon}{\longleftrightarrow}\ $  $\fruit{q}{{\bf V}} \neq \emptyset$ for all $q \in \bset$;
\item [\ding{46}\,]  ${\bf V}$ is {\it regular} $\ {\colon}{\longleftrightarrow}\  V_{a\hspace{0.5pt}^{\frown}n} \subseteq V_a$ for all $a \in \btree$ and $n \in \omega$;
 \item [\ding{46}\,]  ${\bf V}$ has {\it strict branches} $\ {\colon}{\longleftrightarrow}\ |\fruit{q}{{\bf V}}|=1$ for all $q \in \bset$;
 \item[\ding{46}\,] ${\bf {V}}$ is {\it open} on $\ll{X},\tau\rr\ {\colon}{\longleftrightarrow}\ {V}_{a}\in\tau$  for all ${a} \in \btree$;
  \item[\ding{46}\,] ${\bf {V}}$ is {\it semi-open} on $\ll{X},\tau\rr\ {\colon}{\longleftrightarrow}\ {V}_{a} \subseteq \cl{\int{{V}_{a}}{\tau}}{\tau}$  for all ${a} \in \btree$.
 \end{itemize}
\end{deff}

\begin{deff}[Definition 20 in~\cite{PATRAKEEV2023108410}]\mbox{}
 \begin{itemize}
   \item [\ding{46}\,] A {\it $\pi$-net} Souslin scheme on a space ${X}$ is a Souslin scheme ${\bf {V}} = \langle {V}_{a} \rangle_{{a} \in \btree}$ such that $\mathsf{flesh}({\bf {V}})\subseteq{X}$ and the family $\{{V}_{b}: {a}\sqsubseteq{b}\}$ is a $\pi$-net for the subspace ${V}_{a}$ of ${X}$ for all ${a} \in \btree$.
   \item [\ding{46}\,] A {\it $\pi$-base} Souslin scheme on a space ${X}$ is an open $\pi$-net Souslin scheme on ${X}$.
 \end{itemize}
\end{deff}

\begin{deff}[Definition~3.4 in~\cite{patrakeev2015metrizable}]\label{def.luz.p.base}
 A \textit{Lusin $\hspace{1pt}\mathsurround=0pt\pi$-base} for a space $\ll{X},\tau\rr$ is an open Souslin scheme $\langle V_a \rangle_{a \in \btree}$ on $X$ that partitions ${X}$, has strict branches, and such that
  \begin{itemize}
 \item[(L6)\,]
 $\mathsurround=0pt
   \forall{x}\in{X}\enskip\forall{U}\in\tau({x})$

   $\mathsurround=0pt
   \exists {a}\in\btree\enskip
   \exists{n}\in\omega\ $
   \begin{itemize}
   \item[\ding{226}\,]
  $\mathsurround=0pt
  {x} \in {V}_{a}\enskip$ and
   \item[\ding{226}\,]
  $\mathsurround=0pt
  \bigcup_{{i}\geqslant{n}}{V}_{{a}\hspace{0.5pt}^{\frown}{i}}\subseteq {U}.$
   \end{itemize}
  \end{itemize}
\end{deff}

\begin{nota}\label{not.S.N}\mbox{}
  \begin{itemize}
  \item [\ding{46}\,]
  $\mathsurround=0pt
  \mathbf{S}$ $\coloneq$ the \textit{standard Lusin scheme} $\coloneq$ the Souslin scheme $\langle {S}_{a}\rangle_{{a}\in\btree}$ such that

  $\mathsurround=0pt
  {S}_{a}=\{{p}\in\bset:{a}\sqsubseteq{p}\}$
  for all ${a}\in\btree.$
  \end{itemize}
\end{nota}

\begin{rema}\label{rem.baire.space}
  \begin{itemize}
  \item [(a)]
 The family $\{{S}_{a}:{a}\in\btree\}$ is a base for the Baire space.
  \item [(b)]
 The standard Lusin scheme is a Lusin $\pi\mathsurround=0pt$-base for the Baire space.
  \end{itemize}\hfill$\qed$%
\end{rema}

\begin{deff}[Definition 6 in~\cite{PATRAKEEV2023108410}]\label{def_of_pi_sp}
 A space ${X}$ is a $\pi$-\textit{space} iff there exists an open Souslin scheme $\langle V_a \rangle_{a \in \btree}$ on $X$ that partitions ${X}$, has strict branches, and such that the family $\{{V}_{a} : {a} \in \btree\}$ is a $\pi$-base for ${X}$.
\end{deff}

\begin{rema}\label{pi.base}
  If  $\langle V_a \rangle_{a \in \btree}$ is a Lusin $\pi\mathsurround=0pt$-base for a space ${X},$ then the family $\{{V}_{a}:{a}\in\btree\}$ is a $\pi\mathsurround=0pt$-base for ${X}$.

  It follows that every space with a Lusin $\pi$-base is a $\pi$-space.\hfill$\qed$%
\end{rema}

\begin{deff}
 A space $\langle \bset, \tau \rangle$ is a {\it standard $\pi$-space} if $\taun \setminus \{\varnothing\}$ is a $\pi$-base for $\langle \bset, \tau \rangle$.
\end{deff}

The Baire space is a standard $\pi$-space. It follows form Proposition~9(a,c) in~\cite{PATRAKEEV2023108410} that
\begin{rema}\label{rem.st.pi.space}\mbox{}
  \begin{itemize}
  \item [\ding{226}\,]
  Every standard $\pi$-space is a $\pi$-space.
  \item [\ding{226}\,]
  Every  $\pi$-space is homeomorphic to some standard $\pi$-space.\hfill$\qed$%
  \end{itemize}
\end{rema}

\section{Description of open images of spaces with a Lusin $\pi$-base}
\begin{nota}
  Let ${\mathbf{V}} = \ll {V}_{a} \rr_{{a}\in\btree}$ be a Souslin scheme, ${b}\in\btree,$ and ${k}\in\omega.$  Then

  \begin{itemize}
  \item [\ding{46}\,]
 $\mathsurround=0pt
 \widetilde{{V}}^{{k}}_{b}\:\coloneq\:
 \bigcup_{{j}\geqslant{k}}{V}_{{b}\hspace{0.5pt}^{\frown}{j}}$;
  \item [\ding{46}\,]
 $\mathsurround=0pt
 \shoot{{b}}{{\bf {V}}}\:\coloneq\:
 \big\{\widetilde{{V}}^{{k}}_{b}:{k}\in\omega\big\}$;
  \item [\ding{46}\,]
 $\mathsurround=0pt
 \gamma\inn{U}
 \ {\colon}{\longleftrightarrow}\
 \exists{G}\in\gamma\:[{G}\subseteq{U}]$.
  \end{itemize}
\end{nota}

\begin{nota} Let $\ll {X}, \tau \rr$ be a space, ${\mathbf{V}}$ a Souslin scheme, ${p}\in\bset$, ${U} \subseteq {X}$, and ${x}\in{X}$. Then
\begin{itemize}
 \item [\ding{46}\,] $p \xrightarrow{\bf {V}}{U} \quad{\colon}{\longleftrightarrow}\quad\exists {L} \in [\omega]^\omega\enskip \forall {n}\in {L}\ [ \shoot{{p} \upharpoonright {n}}{{\bf {V}}}\inn{U}]$;
 \item [\ding{46}\,] ${p}\xrightarrow{\mathbf{V}\!,\tau} {x}\quad{\colon}{\longleftrightarrow}\quad\forall {U}\in\tau({x})\ [p \xrightarrow{\bf {V}} U]$.
\end{itemize}
\end{nota}

\begin{nota}
  Let ${\mathbf{V}} = \ll {V}_{a} \rr_{{a}\in\btree}$ be a Souslin scheme and ${f}$ is a function. Then
  \begin{itemize}
  \item [\ding{46}\,]
 $\mathsurround=0pt {f}\l\mathbf{V}\r$ is the Souslin scheme $\ll{U}_{a}\rr_{{a}\in\btree}$ such that
 ${U}_{a}={f}[{V}_{a}]$ for all ${a}\in\btree$.
  \end{itemize}
\end{nota}

\begin{rema} \label{vrost_pres_map}
 Suppose that $\gamma$ a family of subsets of ${X}$, ${U} \subseteq {X}$, ${f}$ is a function whose domain equals ${X}$, ${\mathbf{V}}$ is a Souslin scheme such that $\mathsf{flesh}({\bf {V}})\subseteq{X}$, ${b}\in\btree$, and ${p}\in\bset$. Then
 \begin{itemize}
  \item [\ding{46}\,]
 if $\gamma\inn{U}$, then $\{{f}[{A}] : {A} \in \gamma\} \inn{{f}[{U}]}$;
  \item [\ding{46}\,]
 if $\shoot{{b}}{{\bf {V}}} \inn{U}$, then $\shoot{{b}}{{f}\l{\bf {V}}\r} \inn{{f}[{U}]}$;
  \item [\ding{46}\,]
 if $p \xrightarrow{\bf {V}} {U}$, then $p \xrightarrow{{f}\l{\bf {V}}\r} {f}[{U}]$.\hfill$\qed$%
  \end{itemize}
\end{rema}

\begin{rema}\label{appl_pres_con_fun}
 Suppose that ${f}$ is a continuous function from a space $\ll {X}, \tau \rr$ to a space $\ll {Y}, \sigma \rr$, ${\bf {V}}$ is a Souslin scheme such that $\mathsf{flesh}({\bf {V}})\subseteq{X}$, ${p} \in \bset$, and ${x} \in {X}$. Then
  \begin{itemize}
  \item [\ding{46}\,]
  if ${p}\xrightarrow{\mathbf{V}\!,\tau} {x}$,
  then ${p}\xrightarrow{{f}\l\mathbf{V}\r,\sigma} {f}({x})$.
  \end{itemize}
\end{rema}

\begin{proof}
 Suppose that ${p} \in \bset$, ${x} \in \fruit{p}{\bf {V}}$, and ${p}\xrightarrow{\mathbf{V}\!,\tau} {x}$. Take $U \in \sigma(f(x))$. Since $f$ is continuous, we see that
 \begin{equation*}
  p \xrightarrow{\bf {V}} {f}^{-1}[U].
 \end{equation*}
 Then it follows from  Remark \ref{vrost_pres_map} that
 \begin{equation*}
  p \xrightarrow{{f}\l{\bf {V}}\r} U.
 \end{equation*}
\end{proof}

\begin{rema}\label{rem.L6'}
  The clause (L6) of the definition of a Lusin $\pi\mathsurround=0pt$-base is equivalent to the following:
  \begin{itemize}
  \item[(L6')]
 $\mathsurround=0pt
 \forall{x}\in{X}\enskip\forall{U}\in\tau({x})$

 $\mathsurround=0pt
 \exists {a}\in\btree\ $
 such that
 \begin{itemize}
 \item[\ding{226}\,]
 $\mathsurround=0pt
 {x} \in {V}_{a}\enskip$ and
 \item[\ding{226}\,]
 $\mathsurround=0pt
 \shoot{{a}}{{\bf {V}}}\inn {U}.$\hfill$\qed$%
 \end{itemize}
  \end{itemize}
\end{rema}

\begin{lemm}\label{lpb_appl}
 If ${\bf {V}} = \ll {V}_{a} \rr_{{a} \in \btree}$ is a Lusin $\pi$-base for a space $\ll {X}, \tau \rr$ and ${p}\in\bset$,     then ${p}\xrightarrow{\mathbf{V}\!,\tau} {x}$ for all $x \in \fruit{p}{\bf {V}}$.
\end{lemm}

\begin{proof}
 Let ${p} \in \bset$, ${x} \in \fruit{p}{\bf {V}}$, ${U} \in \tau({x})$, and ${i}\in\omega$.
 We have ${U} \cap {V}_{{p} \upharpoonright {i}}\in \tau({x})$, so it follows from (L6')  that
 there exists ${a}_{i}\in\btree$ such that
 \begin{equation*} \label{L_prop}
   {x}\in{V}_{{a}_{i}}\ \text{ and }\ \shoot{{a}_{i}}{{\bf {V}}}\inn {U} \cap {V}_{{p} \upharpoonright {i}}.
 \end{equation*}
 Since ${\bf {V}}$ is a Lusin $\pi$-base, it follows that ${a}_{i}={p} \upharpoonright {n}_{i}$ for some ${n}_{i}\in\omega$, and then we have ${n}_{i}\geqslant{i}$ and
 $\shoot{{p} \upharpoonright {n}_{i}}{{\bf {V}}}\inn {U}$.
 The set $L \coloneq \{{n}_{i} : {i} \in \omega\}$ is infinite and we have
 \begin{equation*}
  \forall n \in L\ [\shoot{{p} \upharpoonright {n}}{{\bf {V}}}\inn {U}].
 \end{equation*}
\end{proof}

\begin{deff}
 An \textit{$\alpha$-scheme} for a space $\ll {X}, \tau \rr$ is an open complete Souslin scheme ${\mathbf{V}} = \ll {V}_{a} \rr_{{a}\in\btree}$ on $\ll X, \tau \rr$ that covers ${X}$ and such that:
 \begin{itemize}
  \item[(S1)] For all ${a}\in\btree$ and all ${x}\in{V}_{a}$, there exists ${p}\in{S}_{a}$ such that

  ${x}\in\fruit{{p}}{\mathbf{V}}\ \text{ and }\ {p}\xrightarrow{\mathbf{V}\!,\tau} {x}$;
  \item[(S2)] $\forall {p}\in\bset\enskip \exists {x}\in\fruit{{p}}{\mathbf{V}}\enskip[{p}\xrightarrow{\mathbf{V}\!,\tau} {x}]$.
 \end{itemize}
\end{deff}

It follows from Lemma~\ref{lpb_appl} that
\begin{rema}
  Every Lusin $\pi$-base for a space ${X}$ is an $\alpha$-scheme for a space ${X}$.\hfill$\qed$%
\end{rema}
Note also that if $\langle V_a \rangle_{a \in \btree}$ is an $\alpha$-scheme for a space $X$, then the family $\{{V}_{a} : {a} \in \btree\}$ is a $\pi$-base for ${X}$.

\begin{prop} \label{LP_pi_base}
 If \,${\bf {W}} = \langle {W}_{a} \rangle_{{a} \in \btree}$ \!is an $\alpha$-scheme for a space $\ll X, \tau\rr$, then ${\bf {W}}$ is a $\pi$-base Souslin scheme on $\ll X, \tau\rr$.
\end{prop}

\begin{proof}
    Take ${a} \in \btree$ and nonempty ${U} \in \tau$ such that ${U}\subseteq{W}_{a}$. Take $x \in {U}$. By (S1) we can find ${p} \in {S}_{a}$ such that ${x}\in\fruit{{p}}{\bf {W}}$ and ${p}\xrightarrow{\mathbf{W}\!,\tau} {x}$. Since ${p}\xrightarrow{\mathbf{W}\!,\tau} {x}$ and ${U} $ is an open neighbourhood of ${x}$, we see that there exists $n > \lh({a})$ such that $\shoot{{p}\uph{n}}{\mathbf{W}} \inn{{U} }$. So there exists $b \in \btree$ such that $a \sqsubset b$ and ${W}_{b} \subseteq {U}$.
\end{proof}

\begin{teo}\label{desc_open_LPB}
 A space ${X}$ is a continuous open image of a space with a Lusin $\pi$-base if and only if ${X}$ has an $\alpha$-scheme and $\:|X| \leq \mathfrak{c}$.
\end{teo}

\begin{deff}
 Let $\ll {X}, \tau \rr$ be a space. A \textit{ramose $\alpha$-scheme} for a space $\ll {X}, \tau \rr$ is an $\alpha$-scheme ${\mathbf{V}} = \ll {V}_{a} \rr_{{a}\in\btree}$ for $\ll {X}, \tau \rr$  such that:
 \begin{itemize}
  \item[(RS1)] For all ${a}\in\btree$ and all ${x}\in{V}_{a}$, there exist continuum many ${p}\in{S}_{a}$ such that

  ${x}\in\fruit{{p}}{\mathbf{V}}\ \text{ and }\ {p}\xrightarrow{\mathbf{V}\!,\tau} {x}$.
 \end{itemize}
\end{deff}

\begin{nota}
 Suppose that ${\bf {V}} = \langle {V}_{a} \rangle_{{a} \in \btree}$ is a Souslin scheme and ${g} \colon \omega \to \omega$. Then ${\bf {V}}^{{g}} = \langle {V}^{{g}}_{a} \rangle_{{a} \in \btree}$ is a Souslin scheme such that ${V}^{{g}}_{a} \coloneq {V}_{{g} \circ {a}}$ for all ${a} \in \btree$.
\end{nota}

\begin{lemm} \label{LP_to_strongLP}
    If a space has an $\alpha$-scheme, then it has a ramose $\alpha$-scheme.
\end{lemm}

\begin{proof}
  Let $\,{\mathbf{V}} = \ll {V}_{a} \rr_{{a}\in\btree}$ be an $\alpha$-scheme for a space $\ll {X}, \tau\rr$.
  Take a function $g\colon \omega \to \omega$ such that $|{g}^{-1}(n)| = 2$ for all ${n} \in \omega$.
  We shall show that ${\bf V}^{g}$ is a ramose $\alpha$-scheme for a space $\ll {X}, \tau\rr$.

  From \cite[Lemma~32]{PATRAKEEV2023108410} it follows that ${\bf V}^{g}$ is an open complete Souslin scheme on $\ll X, \tau \rr$ that covers ${X}$. 
  Condition (S2) for ${\bf V}^{g}$ follows from condition (S2) for ${\bf V}$.

  Let us check condition (RS1) for ${\bf V}^{g}$; note that (S1) follows from it. Take ${a} \in \btree$ and ${x} \in {V}^{g}_{a} = {V}_{{g}\circ{a}}$. From (S1) for ${\bf V}$ it follows that there exists a branch ${p} \in {S}_{{g} \circ {a}}$ such that ${x}\in\fruit{{p}}{{\bf V}}$ and ${p}\xrightarrow{\mathbf{V}\!,\tau} {x}$.
  Consider the set
  $$
  {F}\coloneq\{{q}\in{S}_{a}:{g}\circ{q}={p}\}.
  $$
  This set has cardinality of continuum. For all ${q}\in{F}$, we have 
  \[
  {x}\in\fruit{{p}}{\bf V}=
  \bigcap_{{n}\in\omega}{V}_{{p}\upharpoonright{n}}=
  \bigcap_{{n}\in\omega}{V}_{({g}\circ{q})\upharpoonright{n}}=
  \bigcap_{{n}\in\omega}{V}_{{g}\circ({q}\upharpoonright{n})}=
  \bigcap_{{n}\in\omega}{V}^{g}_{{q}\upharpoonright{n}}=
  \fruit{{q}}{{\bf {V}}^{\!{g}}}.
  \]

  It remains to show that ${q}\xrightarrow{\mathbf{V}^{g}\!,\tau} {x}$ for all ${q}\in{F}$.
  We have ${p}\xrightarrow{\mathbf{V}\!,\tau} {x}$, so ${p} \xrightarrow{\bf {V}} U$ for all ${U}\in\tau({x})$.
  This means
  \[
   \forall {U}\in\tau({x})\,\exists {L} \in [\omega]^\omega\,\forall{{n}\in {L}}\
   [\,\shoot{{p} \upharpoonright {n}}{{\bf {V}}}\inn {U}\,].
  \]
  Since  ${g}\circ{q}={p}$, we have
  \[
   \forall {U}\in\tau({x})\,\exists {L} \in [\omega]^\omega\,\forall{{n}\in {L}}\
   [\,\shoot{{g}\circ({q} \upharpoonright {n})}{{\bf {V}}}\inn {U}\,].
  \]
  Note that for every ${b} \in \btree$ and every set ${U}$, if $\shoot{{g}\circ{b}}{{\bf {V}}}\inn {U}$,  then
  $\shoot{b}{{\bf {V}}^{g}}\inn {U}$.
  It follows that
  \[
   \forall {U}\in\tau({x})\,\exists {L} \in [\omega]^\omega\,\forall{{n}\in {L}}\
   [\,\shoot{{q} \upharpoonright {n}}{{\bf {V}}^{g}}\inn {U}\,].
  \]
  This means ${q}\xrightarrow{\mathbf{V}^{g}\!,\tau} {x}$.
\end{proof}

\begin{nota}\label{not.S.N}\mbox{\ } Let ${\bf {V}} = \langle {V}_{a} \rangle_{{a} \in \btree}$ be a Souslin scheme. Then
  \begin{itemize}
 \item [\ding{46}\,]  $\mathsf{branches}_{\bf {V}} ({x}) \coloneq \{q \in \bset : x \in \fruit{q}{{\bf V}}\}$.
  \end{itemize}
\end{nota}

\begin{deff}[Definition~21 in~\cite{PATRAKEEV2023108410}]
 A \emph{selector} on a Souslin scheme ${\bf {V}}$ is a surjection  ${f}\colon\bset\to\mathsf{flesh}({\bf {V}})$ such that for all ${x} \in \mathsf{flesh}({\bf {V}})$, the preimage ${f}^{-1}(x)$ is a dense subset of the subspace $\mathsf{branches}_{\bf {V}}({x})$ of the Baire space.
\end{deff}

If a Souslin scheme ${\bf {V}}$ has  strict branches and covers a set ${X}$, then the function ${f} \colon \bset \to {X}$ such that $\{{f}({p})\} = \fruit{p}{{\bf {V}}}$ is a selector on ${\bf {V}}$.

A less trivial example of a selector can be obtained as follows. Let ${f}$ be a continuous surjection from the Baire space onto a space ${X}$. Let ${V}_{a} \coloneq {f}[{S}_{a}]$ for all ${a} \in \btree$. Then ${f}$ is a selector on $\ll {V}_{a} \rr_{{a} \in \btree}$.

\begin{lemm}[Lemma~22 in~\cite{PATRAKEEV2023108410}]\label{main_property_of_fiber_scheme}
 Let ${\bf {V}} = \langle {V}_{a} \rangle_{{a} \in \btree}$ be a Souslin scheme that covers $\mathsf{flesh}({\bf {V}})$ and let ${f}$ be a selector on ${\bf {V}}$. Then ${f}[{S}_{a}] = {V}_{a}$ for all ${a} \in \btree$.\hfill \qed
\end{lemm}

\begin{lemm} \label{strongLP_has_selector}
 Let ${\bf V} = \ll {V}_{a} \rr_{{a} \in \btree}$ be a ramose $\alpha$-scheme for a space $\ll {X}, \tau \rr$. If $|{X}| \leq \mathfrak{c}$, then there exists a selector ${f}$ on ${\bf V}$ such that ${p}\xrightarrow{\mathbf{V}\!,\tau} {x}$ for all ${x} \in {X}$ and all ${p} \in {f}^{-1}({x})$.
\end{lemm}

\begin{proof}
  For all ${x} \in {X}$ and ${a} \in \btree$ denote $$
    {P}({x}, {a}) \coloneq \{{p} \in \mathsf{branches}_{\bf {V}}({x}) \cap {S}_{a} \colon {p}\xrightarrow{\mathbf{V}\!,\tau} {x}\}.
$$

    Take ${x} \in {X}$ and ${a} \in \btree$. Note that if $\mathsf{branches}_{\bf {V}}({x}) \cap {S}_{a} \neq \varnothing$, then ${x} \in {V}_{a}$, and so, by (RS1), $|{P}({x}, {a})| = \mathfrak{c}$. Thus we have
    \begin{equation} \label{ZOV}
        \mathsf{branches}_{\bf {V}}({x}) \cap {S}_{a} \neq \varnothing \to |{P}({x}, {a})| = \mathfrak{c}
    \end{equation}

  By transfinite recursion on ${X}$ well-ordered in the type of its cardinality, it is easy to build an indexed family $\langle {Q}_{x} \rangle_{{x} \in X}$ such that
 \begin{itemize}
 \item[\ding{226}\,] ${Q}_{x}$ is a countable dense subset of $\mathsf{branches}_{\bf {V}}({x})$ for all ${x} \in X$,
  \item[\ding{226}\,] ${Q}_{x} \cap {Q}_{y} = \emptyset$ for all ${x} \neq {y} \in \mathsf{flesh}({\bf {V}})$, and
  \item[\ding{226}\,] ${p}\xrightarrow{\mathbf{V}\!,\tau} {x}$ for all ${x} \in X$ and all ${p} \in {Q}_{x}$. 
 \end{itemize}
Let ${x} \in X$. Suppose that the sets ${Q}_{y}$ have been chosen for all ${y}$ before ${x}$. We shall find a countable dense set ${Q}_{x}$ in $\mathsf{branches}_{\bf {V}}({x})$ such that ${Q}_{x} \cap {Q}_{y} = \varnothing$ for all ${y}$ before ${x}$. Note that $\{{S}_{a} \cap \mathsf{branches}_{\bf {V}}({x}) : {a} \in \btree\}$ is a countable base for  $\mathsf{branches}_{\bf {V}}({x})$. From (\ref{ZOV}) it follows that for every ${a} \in \btree$, if ${S}_{a} \cap \mathsf{branches}_{\bf {V}}({x}) \neq \varnothing$, then there exists ${p}_{a} \in {S}_{a} \cap \mathsf{branches}_{\bf {V}}({x})$ such that ${p}_{a}\xrightarrow{\mathbf{V}\!,\tau} {x}$ and ${p}_{a} \nin {Q}_{y}$ for all ${y}$ before ${x}$. Then ${Q}_{x} \coloneq \{{p}_{a}\ \colon\ {a} \in \btree \}$ satisfies required conditions.

Now we can construct a selector ${f}\colon\bset\to{X}$ on ${\bf {V}}$. If ${p}
 \in{Q}_{x}$ for some ${x}\in {X}$, then set ${f}({p})\coloneq{x}$.  If ${p}\nin\bigcup_{{x}\in{X}}{Q}_{x}$, then using (S2) choose ${f}({p})\in\fruit{p}{{\bf {V}}}$ such that ${p}\xrightarrow{\mathbf{V}\!,\tau} {f}({p})$. It is easy to see that ${f}$ is a selector on ${\bf {V}}$.
\end{proof}

\begin{deff}[Definition~23 in~\cite{PATRAKEEV2023108410}]
 Let $\langle {X}, \tau \rangle$ be a space, ${\bf {V}}$ a Souslin scheme that covers ${X}$, and ${f}$ a selector on ${\bf {V}}$. Then $\sigma_{\tau, {f}}$ is the topology on $\bset$ generated by the subbase $\{{f}^{-1}[{U}]: {U} \in \tau\}\cup \{{S}_{a}: {a} \in \btree\}$.
\end{deff}

\begin{lemm}[Lemma~26 in~\cite{PATRAKEEV2023108410}] \label{fXV_is_open}
 Let $\langle {X}, \tau \rangle$ be a space, ${\bf {V}}$ an open Souslin scheme on $\langle {X}, \tau \rangle$ that covers ${X}$, and ${f}$ a selector on ${\bf {V}}$. Then ${f} \colon \langle \bset, \sigma_{\tau, {f}} \rangle \to \langle {X}, \tau \rangle$ is a continuous open surjection.\hfill\qed
\end{lemm}

\begin{lemm} \label{surjection}
 Let ${f} \colon {A} \to {X}$ be a surjection, ${S} \subseteq {A}$, ${V} \subseteq {X}$, and ${f}[{S}] = {V}$. Then ${f}\big[{f}^{-1}[{U}] \cap{S}\big] = {U} \cap {V}$ for all ${U} \subseteq {X}$. \hfill \qed 
\end{lemm}

\begin{lemm}\label{has_lpb}
 Let ${\bf {V}} = \ll {V}_{a} \rr_{{a} \in \btree}$ be an open complete Souslin scheme on $\ll {X}, \tau \rr$ that covers ${X}$ and let ${f}$ be a selector on ${\bf V}$ such that ${p}\xrightarrow{\mathbf{V}\!,\tau} {x}$ for all ${x} \in {X}$ and all ${p} \in {f}^{-1}({x})$. Then ${\bf {S}}$ is a Lusin $\pi$-base for $\ll \bset, \sigma_{\tau,{f}} \rr$.
\end{lemm}

\begin{proof}
 Using Remark \ref{rem.L6'} it is enough to check (L6'). Note that the family   $\{{f}^{-1}[{U}]\cap {S}_{a}:{U}\,{\in}\,\tau,\: {a}\,{\in}\,\btree\}$ is a base for the space $\ll\bset,\sigma_{\tau, {f}}\rr$. 
 Take ${x} \in \bset$ and its arbitrary base neighbourhood ${f}^{-1}[U] \cap {S}_{b}$. 
 Since ${x}\xrightarrow{\mathbf{V}\!,\tau} {f}({x})$ and ${f}({x}) \in {f}[{f}^{-1}[U] \cap {S}_{b}] = {U} \cap {V}_{b}$ (the equality follows from Lemma \ref{main_property_of_fiber_scheme} and Lemma \ref{surjection}), we see that ${x}\xrightarrow{\mathbf{V}} {U} \cap {V}_{b}$, and so there exists $n \geq \lh({b})$ such that $\shoot{{x} \upharpoonright {n}}{{\bf {V}}}\inn {U} \cap {V}_{b}$. 
 Take ${a}\coloneq{x}\upharpoonright{n}$; then ${x} \in {S}_{a}$. 
 Note that ${S}_{({x} \upharpoonright {n})\hspace{0.5pt}^{\frown}{k}}\subseteq{S}_{b}$ for all ${k}\in\omega$ because ${x}\in{S}_{b}$ and $n \geq \lh({b})$. It follows that $\shoot{a}{{\bf {S}}}\inn {f}^{-1}[{U}] \cap {S}_{b}$.
\end{proof}

\begin{lemm} \label{has_lp_sch}
 Let ${\bf {V}} = \ll {V}_{a} \rr_{{a} \in \btree}$ be a Lusin $\pi$-base for a space $\ll {X}, \tau \rr$ and ${f}\colon \ll {X}, \tau \rr \to \ll {Y}, \sigma \rr$ be an open continuous surjection. Then ${f}\l{\bf {V}}\r$ is an $\alpha$-scheme for $\ll {Y}, \sigma \rr$.
\end{lemm}

\begin{proof}
 Let us check (S1), that is, prove 
 \begin{equation}
  \forall {a}\in\btree\enskip \forall {x}\in{f}[{V}_{a}]\enskip \exists {p}\in{S}_{a} \ [ {x}\in\fruit{{p}}{{f}\l{\bf {V}}\r} \text{ and } {p}\xrightarrow{{f}\l\mathbf{V}\r,\sigma} {x}].
 \end{equation}
 Take ${a}\in\btree$ and ${x}\in{f}[{V}_{a}]$. There exists ${y} \in {V}_{a}$ such that ${f}({y}) = {x}$. Take ${p} \in \bset$ such that $\{{y}\}=\fruit{{p}}{{\bf {V}}}$, then ${x} \in \fruit{{p}}{{f}\l{\bf {V}}\r}$ and ${p}\in{S}_{a}$. By Lemma \ref{lpb_appl}, ${p}\xrightarrow{\mathbf{V},\tau} {y}$, so by Remark~\ref{appl_pres_con_fun} we have
 \begin{equation}
  {p}\xrightarrow{{f}\l\mathbf{V}\r,\sigma} {x}.
 \end{equation}

 Now check (S2), that is,  prove 
 \begin{equation}
  \forall {p}\in\bset\enskip \exists {x}\in\fruit{{p}}{{f}\l{\bf {V}}\r}\ [{p}\xrightarrow{{f}\l\mathbf{V}\r,\sigma} {x}].
 \end{equation}
 Take ${p} \in \bset$. Consider ${y} \in {X}$ such that $\{{y}\}=\fruit{{p}}{{\bf {V}}}$, then ${f}({y}) \in \fruit{{p}}{{f}\l{\bf {V}}\r}$. From Lemma \ref{lpb_appl} and Remark~\ref{appl_pres_con_fun} it follows that
 \begin{equation}
  {p}\xrightarrow{{f}\l\mathbf{V}\r,\sigma} {f}({y}).
 \end{equation}
\end{proof}

\begin{proof}[\textbf{Proof of Theorem~\ref{desc_open_LPB}}]
    Suppose that a space $\ll {X}, \tau \rr$ is a continuous open image of a space with Lusin $\pi$-base, then from Lemma \ref{has_lp_sch} it follows that there exists an $\alpha$-scheme for $\ll {X}, \tau \rr$.

    Suppose that $|X| \leq \mathfrak{c}$ and there exists an $\alpha$-scheme for $\ll {X}, \tau \rr$. Then from Lemma \ref{LP_to_strongLP} it follows that there exists a ramose $\alpha$-scheme ${\bf V} $ for $\ll {X}, \tau \rr$. Now from Lemma \ref{strongLP_has_selector} we see that there exists a selector ${f}$ on ${\bf V}$ such that  ${p}\xrightarrow{\mathbf{V}\!,\tau} {x}$ for all ${x} \in {X}$ and  all ${p} \in {f}^{-1}({x})$. From Lemma \ref{has_lpb} it follows that $\ll \bset, \sigma_{\tau,{f}} \rr$ has a Lusin $\pi$-base and from Lemma \ref{fXV_is_open} it follows that $f \colon \ll \bset, \sigma_{\tau,{f}} \rr \to \ll {X}, \tau \rr$ is a continuous open surjection.
\end{proof}

\section{An example of a zero-dimensional $\pi$-space without an $\alpha$-scheme}

In this section we will prove that the class of open images of spaces with a Lusin $\pi$-base is a proper subclass of the class of open images of $\pi\nos$-spaces:

\begin{teo}\label{teo.main}
 There exists a zero-dimensional $\pi$-space $X$ such that $X$ is not a continuous open image of a space with a Lusin $\pi$-base.
\end{teo}

\begin{lemm}\label{shoot2}
Let ${\bf {V}}$ be a Souslin scheme and ${a} \in \btree$. Then:

1. If $\shoot{a}{{\bf {V}}} \ninn{A}$ and ${A}\supseteq{B}$, then $\shoot{a}{{\bf {V}}} \ninn{B}$.

2. $\shoot{a}{{\bf {V}}} \inn{A} \cap{B}$ if and only if $\shoot{a}{{\bf {V}}} \inn{A}$ and $\shoot{a}{{\bf {V}}} \inn{B}$.
\hfill\qed
\end{lemm}

\begin{prop} \label{stand_pi_sp}
 Let ${F}: \bset \to \taun \setminus \{\varnothing\}$ be such that ${x} \in \cl{{F}({x})}{\taun}$ for all ${x} \in \bset$. Then there exists a standard $\pi$-space $\ll \bset, \tau \rr$ such that $\{{x}\} \cup {F}({x}) \in \tau$ for all ${x} \in \bset$. 
 
 Moreover, if $\cl{{F}({x})}{\taun} = \{{x}\} \cup {F}({x})$ for all ${x} \in \bset$, then $\ll \bset, \tau \rr$ is zero-dimensional.
\end{prop}

\begin{proof}
 Let $\tau$ be the topology on $\bset$ generated by the subbase $\{{S}_{a} : {a} \in \btree\} \cup \{\{{x}\} \cup {F}({x}) : {x} \in \bset\}$. First we need to show that  $\taun\setminus\{\varnothing\}$ is a $\pi\nos$-base for $\ll \bset, \tau \rr$.
 Take ${b} \in \btree$, ${n}\in\omega$, and ${x}_0, \dots, {x}_{{n}-1}$ are ${n}$ different points in $\bset$ such that
 \begin{equation} \label{not_noth}
  S_b \cap \bigcap_{{i}\in{n}}(\{{x}_{i}\}\cup{F}({x}_{i})) \neq \varnothing.
 \end{equation}
 We must prove that 
 \begin{equation} \label{main_purpose_of_pi}
  \int{S_b \cap \bigcap_{{i}\in{n}}(\{{x}_{i}\}\cup{F}({x}_{i}))}{\taun} \neq \varnothing.
 \end{equation}
 
 Denote
 \begin{equation}
  {A} \coloneq \{{i} \in {n} : {x}_{i} \in \bigcap_{{j}\in{{n}\setminus\{{i}\}}}{F}({x}_{j})\}.
 \end{equation}
 Note that $\bigcap_{{j}\in{{n}\setminus\{{i}\}}}{F}({x}_{j})$ is a neighbourhood of ${x}_{i}$ in $\taun$ for all ${i}\in{A}$. 
 Since ${x} \in \cl{{F}({x})}{\taun}$ for all ${x} \in \bset$, it follows that
 \begin{equation} \label{zhopa3}
  {x}_{i} \in \cl{\bigcap_{{j}\in{n}}{F}({x}_{j})}{\taun}\ \text{ for all }  {i} \in {A} .
 \end{equation}
 Also we have
 \begin{equation} \label{zhopa1}
  \bigcap_{{i}\in{n}}\big(\{{x}_{i}\}\cup{F}({x}_{i})\big) \:=\: \{{x}_{i} : {i} \in {A}\} \cup\, \bigcap_{{i}\in{n}}{F}({x}_{i}).
 \end{equation}
 To prove (\ref{zhopa1}), note that if ${x}_{k}\in \bigcap_{{i}\in{n}}(\{{x}_{i}\}\cup{F}({x}_{i}))$, then ${k}\in{A}$.
 
 Let us show  that
 \begin{equation} \label{zhopa2}
  {S}_{b} \cap \bigcap_{{i}\in{n}}{F}({x}_{i}) \neq \varnothing.
 \end{equation}
 Suppose that ${S}_{b} \cap \bigcap_{{i}\in{n}}{F}({x}_{i}) = \varnothing$. Then from (\ref{zhopa1}) and (\ref{not_noth}) it follows that ${S}_{b} \cap \{{x}_{i} : {i} \in {A}\} \neq \varnothing$. Take $x \in {S}_{b} \cap \{{x}_{i} : {i} \in {A}\}$. ${S}_{b}$ is a neighbourhood of ${x}$, so by (\ref{zhopa3}) it follows that ${S}_{b} \cap \bigcap_{{i}\in{n}}{F}({x}_{i}) \neq \varnothing$, a contradiction.
 
 Now, since
 $$
  \bigcap_{{i}\in{n}}{F}({x}_{i}) \in \taun,
 $$
 we see that (\ref{main_purpose_of_pi}) follows from (\ref{zhopa2}). 

 Now suppose that for all  ${x} \in \bset$,
 $$
     \cl{{F}({x})}{\taun} = \{{x}\} \cup {F}({x}).
$$
In this case, every set of the form as in (\ref{not_noth}) is closed in $\ll \bset, \tau \rr$. So, since these sets from a base for $\ll \bset, \tau \rr$, $\ll \bset, \tau \rr$ is zero-dimensional.
\end{proof}

\begin{lemm} \label{scheme_native}
 Let ${\bf {V}}$ be an open Souslin scheme on a standard $\pi$-space $\ll \bset, \tau \rr$. Then ${\bf {V}}$ is semi-open on $\bspace$.
\end{lemm}

\begin{proof}
 We must show that if ${U}\in\tau$, then ${U}\subseteq \cl{\int{{U}}{\taun}}{\taun}$. The set $\int{{U}}{\taun}$ is dense in ${U}$ in the space $\ll \bset, \tau \rr$ because $\taun\setminus\{\varnothing\}$ is a $\pi\nos$-base for $\ll \bset, \tau \rr$. Then $\int{{U}}{\taun}$ is dense in ${U}$ in $\ll \bset, \taun \rr$ because $\taun\subseteq\tau$. 
 \end{proof}

\begin{nota}\label{not.S.N} Let ${\bf {V}} = \langle {V}_{a} \rangle_{{a} \in \btree}$ be a Souslin scheme and $\tau$ a topology. Then

  \begin{itemize}
 \item [\ding{46}\,]
  $\mathsf{Int}({\bf {V}},\tau)$ is a Souslin scheme $\ll {W}_{a} \rr_{{a} \in \btree}$ such that ${W}_{a} = \int{{V}_{a}}{\tau}$.

 \end{itemize}
\end{nota}

\begin{prop}\label{function.F}
There exists a function  ${F}\colon\bset\to\taun\!\setminus\!\{\varnothing\}$ that satisfies the following properties\textup{:}
\begin{itemize}
 \item[1.] For all ${x}$ in $\bset$,  $\cl{{F}({x})}{\taun}={F}({x})\cup\{{x}\}$\textup{.}
 \item[2.] Suppose that \({\mathbf{W}}\) is a regular open $\pi\nos$-base Souslin scheme on $\bspace$ \!and \,\({\mathbf{W}}\) has nonempty leaves. Then there exist ${p},{x}\in\bset$ such that
 \begin{itemize}
  \item[i.] $\fruit{{p}}{\mathbf{W}}=\{{x}\}$,
  \item[ii.] $\shoot{{p}\uph{n}}{\mathbf{W}}\ninn\{{x}\}\cup{F}({x})$ for all ${n}\in\omega$, and
  \item[iii.] for every regular semi-open Souslin scheme $\mathbf{{V}}$ on $\bspace$,

  if $\,\mathsf{Int}({\bf {V}},\taun)=\mathbf{W}$,
  then $\fruit{{p}}{\mathbf{{V}}}=\{{x}\}$.
 \end{itemize}
\end{itemize}
\end{prop}

\begin{proof}
There are at most continuum open Souslin schemes on $\bspace$, so we may assume that
$\{\mathbf{W}^{\alpha}:\alpha<\mathfrak{c}\}$ is the set of all regular open $\pi\nos$-base Souslin schemes on $\bspace$ that have nonempty leaves.

We will build transfinite sequences $\ll{p}^{\alpha}\rr_{\alpha<\mathfrak{c}}$ and
 $\ll{x}^{\alpha}\rr_{\alpha<\mathfrak{c}}$ in $\bset$ and a transfinite sequence  $\ll{U}^{\alpha}\rr_{\alpha<\mathfrak{c}}$ in $\taun\setminus\{\varnothing\}$
such that, for all $\alpha\in\mathfrak{c}$,
  \begin{itemize}
  \item[a1.]
 ${x}^{\alpha}\neq{x}^{\beta}$ for all $\beta\in\mathfrak{c}\setminus\{\alpha\}$;
  \item[a2.]
 $\cl{{U}^{\alpha}}{\taun}={U}^{\alpha}\cup\{{x}^{\alpha}\}$;
  \item[a3.]
 $\fruit{{p}^{\alpha}}{\mathbf{W}^{\alpha}}=\{{x}^{\alpha}\}$;
  \item[a4.]
 $\shoot{{p}^{\alpha}\uph{n}}{\mathbf{W}^{\alpha}}\ninn\{{x}^{\alpha}\}\cup{U}^{\alpha}$ for all ${n}\in\omega$;
  \item[a5.]
 for every regular semi-open Souslin scheme $\mathbf{{V}}$ on $\bspace$,

 if $\mathsf{Int}({\bf {V}},\taun)=\mathbf{W}^{\alpha}$,
 then $\fruit{{p}^{\alpha}}{\mathbf{{V}}}=\{{x}^{\alpha}\}$.
\end{itemize}

It is easy to prove that conditions (a1)--(a5) imply the assertion of Proposition~\ref{function.F}. The transfinite sequence  $\ll{x}^{\alpha}\rr_{\alpha<\mathfrak{c}}$ is injective by (a1), so we may define function ${F}\colon\bset\to\taun\setminus\{\varnothing\}$ as follows: for all $\alpha<\mathfrak{c}$, we set ${F}({x}^{\alpha})\coloneq{U}^\alpha$, and for all ${x}\in\bset\setminus\{{x}^\alpha:\alpha<\mathfrak{c}\}$, we set ${F}({x})\coloneq\bset$. It is straightforward to show that ${F}$ satisfies required conditions.

It rem{a}ins to build the transfinite sequences $\ll{p}^{\alpha}\rr_{\alpha<\mathfrak{c}}$,
$\ll{x}^{\alpha}\rr_{\alpha<\mathfrak{c}}$, and $\ll{U}^{\alpha}\rr_{\alpha<\mathfrak{c}}$; we will build them by recursion on $\alpha<\mathfrak{c}$. Assume that we have chosen  ${p}^\beta$, ${x}^{\beta}$, and ${U}^\beta$ for all $\beta<\alpha$ in such a way that conditions (a1)--(a5) are satisfied.

Let ${a}_{\ll\rr}\coloneq{b}_{\ll\rr}\coloneq{\ll\rr}\in\btree$. We have ${S}_{{a}_{\ll\rr}}\supseteq{W}^{\alpha}_{{b}_{\ll\rr}}$. Since $\mathbf{W}^{\alpha}$ is open and has nonempty leaves, we can choose ${a}_{\ll\rr^{\frown}0}$ and ${a}_{\ll\rr^{\frown}1}$ in $\btree$ such that
\[{W}^{\alpha}_{{b}_{\ll\rr}}\supseteq{S}_{{a}_{\ll\rr^{\frown}0}},\quad
{W}^{\alpha}_{{b}_{\ll\rr}}\supseteq{S}_{{a}_{\ll\rr^{\frown}1}},\]
\[{S}_{{a}_{\ll\rr^{\frown}0}}\cap{S}_{{a}_{\ll\rr^{\frown}1}}=\varnothing,\]
\[{a}_{\ll\rr}\sqsubset{a}_{{\ll\rr}^{\frown}0},\quad\text{and}\quad
{a}_{\ll\rr}\sqsubset{a}_{{\ll\rr}^{\frown}1}.\]
Since $\mathbf{W}^{\alpha}$ is a $\pi\nos$-base Souslin scheme on the  Baire space, we can choose ${b}_{\ll\rr^{\frown}0}$ and ${b}_{\ll\rr^{\frown}1}$ in $\btree$ such that
\[{S}_{{a}_{\ll\rr^{\frown}0}}\supseteq{W}^{\alpha}_{{b}_{\ll\rr^{\frown}0}},\quad
{S}_{{a}_{\ll\rr^{\frown}1}}\supseteq{W}^{\alpha}_{{b}_{\ll\rr^{\frown}1}},\]
\[{b}_{\ll\rr}\sqsubset{b}_{{\ll\rr}^{\frown}0},\quad\text{and}\quad
{b}_{\ll\rr}\sqsubset{b}_{{\ll\rr}^{\frown}1}.\]

Proceeding this way, we will build two indexed families $\ll{a}_{e}\rr_{{e}\in\ctree}$ and $\ll{b}_{e}\rr_{{e}\in\ctree}$ of elements of $\btree$ such that

\begin{itemize}
\item[b1.]
 ${S}_{{a}_{{t}\uph{n}}}\supseteq
 {W}^{\alpha}_{{b}_{{t}\uph{n}}}\supseteq
 {S}_{{a}_{{t}\uph({n}+1)}}$ for all ${t}\in\cset$ and  ${n}\in\omega$;
\item[b2.]
 ${a}_{{t}\uph{n}}\sqsubset{a}_{{t}\uph({n}+1)}$ for all ${t}\in\cset$ and  ${n}\in\omega$;
\item[b3.]
 ${b}_{{t}\uph{n}}\sqsubset{b}_{{t}\uph({n}+1)}$ for all ${t}\in\cset$ and  ${n}\in\omega$;
\item[b4.]
 ${S}_{{a}_{{e}^{\frown}0}}\cap{S}_{{a}_{{e}^{\frown}1}}=\varnothing$ for all ${e}\in\ctree$.
\end{itemize}

Note that
\begin{equation}\label{SeqW}
 \bigcap_{{n}\in\omega}{S}_{{a}_{{t}\uph{n}}} = \bigcap_{{n}\in\omega}{W}^{\alpha}_{{b}_{{t}\uph{n}}} \text{ for all } {t}\in\cset\,.
\end{equation}
Also it follows from (b4) and (b2) that
\begin{equation*}
\bigcap_{{n}\in \omega}{S}_{{a}_{{t}\uph{n}}} \cap \bigcap_{{n}\in \omega}{S}_{{a}_{{u} \uph {n}}} = \varnothing
 \text{ for all } {t}\neq{u}\in\cset\text{ and}
\end{equation*}
\begin{equation*}
 \bigcap_{{n}\in\omega}{S}_{{a}_{{t}\uph{n}}} \text{ is a singleton for all } {t}\in\cset\,.
\end{equation*}

It follows that there exists $\dot{t}\in\cset$ such that
\begin{equation}\label{x_alpha}
 \bigcap_{{n}\in \omega}{S}_{{a}_{\dot{t} \uph {n}}} \neq \{{x}^\beta\} \text{ for all } \beta < \alpha.
\end{equation}
Let
\begin{equation*}
{c}_{n}\coloneq{a}_{\dot{t}\uph{n}}
\qquad\text{and}\qquad
{d}_{n}\coloneq{b}_{\dot{t}\uph{n}}
\end{equation*}
for all ${n}\in\omega$.
Then (b1)--(b4) imply
\begin{itemize}
\item[c1.]
 ${S}_{{c}_{n}}\supseteq
 {W}^{\alpha}_{{d}_{n}}\supseteq
 {S}_{{c}_{{n}+1}}$ for all ${n}\in\omega$;
\item[c2.]
 ${c}_{n}\sqsubset{c}_{{n}+1}$ for all ${n}\in\omega$;
\item[c3.]
 ${d}_{n}\sqsubset{d}_{{n}+1}$ for all ${n}\in\omega$.
\end{itemize}

By (c3) we have
\begin{equation}\label{c4}
 \lh({d}_{n})<\lh({d}_{{n}+1})\text{ for all } {n}\in\omega.
\end{equation}
Recall that a sequence, being a function, is a set of ordered pairs, so (c2) and (c3) says that  ${c}_{n}\subset{c}_{{n}+1}$ and ${d}_{n}\subset{d}_{{n}+1}$ for all ${n}\in\omega$.

Now we can define ${x}^\alpha$ and ${p}^\alpha$. Put
\begin{equation*}{x}^\alpha \coloneq \bigcup_{{n}\in\omega}{c}_{n}\in\bset
\qquad\text{and}\qquad
{p}^\alpha \coloneq \bigcup_{{n}\in \omega}{d}_{{n}}\in\bset.
\end{equation*}
Note that
\begin{equation}\label{def.x_alpha}
\{{x}^\alpha\} = \bigcap_{{n}\in\omega}{S}_{{c}_{n}}.
\end{equation}
It follows from (\ref{x_alpha}) and (\ref{def.x_alpha}) that ${x}^\alpha\neq{x}^\beta$ for all $\beta<\alpha$, so (a1) is satisfied.

By (\ref{def.x_alpha}), (\ref{SeqW}), regularity of the scheme $\mathbf{{W}^\alpha}$, and (\ref{c4}) we have
$$
\{{x}^\alpha\}
=
\bigcap_{{n}\in\omega}{S}_{{c}_{n}}
=
\bigcap_{{n}\in\omega}{W}^{\alpha}_{{d}_{n}}
=
\bigcap_{{n}\in\omega}{W}^{\alpha}_{{p}^\alpha \uph\lh({d}_{n})}
=
\fruit{{p}^{\alpha}}{\mathbf{W}^{\alpha}},
$$
therefore (a3) is satisfied.

Let $\mathbf{{V}}$ be a regular semi-open Souslin scheme on $\bspace$ such that $\mathsf{Int}({\bf {V}},\taun)=\mathbf{W}^{\alpha}$. 
Then by (\ref{def.x_alpha}), (\ref{SeqW}), (\ref{c4}), and (c1), we have
\\
$$
\{{x}^\alpha\}
=
\bigcap_{{n}\in\omega}{S}_{{c}_{n}}
=
\bigcap_{{n}\in\omega}{W}^{\alpha}_{{d}_{n}}
=
\bigcap_{{n}\in\omega} \int{{V}_{{d}_{n}}}{\taun}
\subseteq
$$
$$
\subseteq
\bigcap_{{n}\in\omega}{V}_{{d}_{n}}
=
\bigcap_{{n}\in\omega}{V}_{{p}^\alpha\uph\lh({d}_{n})}
=
\mathsf{fruit}_{\bf {V}}({p}^\alpha)
=
\bigcap_{{n}\in\omega}{V}_{{p}^\alpha\uph\lh({d}_{n})}
=
$$
$$
=
\bigcap_{{n}\in\omega} {V}_{{d}_{n}}
\subseteq
\bigcap_{{n}\in\omega}
\cl{\int{{V}_{{d}_{n}}}{\taun}}{\taun}
=
\bigcap_{{n}\in\omega}\cl{{W}^{\alpha}_{{d}_{n}}}{\taun}
\subseteq
$$
$$
\subseteq
\bigcap_{{n}\in\omega}\cl{{S}_{{c}_{n}}}{\taun}
=
\bigcap_{{n}\in\omega} {S}_{{c}_{n}}
=
\{{x}^\alpha\},
$$
so (a5) is satisfied.

It remains to define ${U}^\alpha$ and prove (a2) and (a4). We will build two families: an indexed family
$$
\ll {E}^{n}_{k}:{k}\in\omega,{n}\in\omega\cup\{-1\}\rr
$$
of subsets of $\omega$ and an indexed family
$$
\ll {H}^{n}_{k}:{k}\in\omega,{n}\in\omega\cup\{-1\}\rr
$$
of subsets of $\bset$. These families will possess the following properties:

\begin{itemize}
\item[d1.]
 ${H}^{n}_{k} = \bigcup_{{j}\in{E}^{n}_{k}} {S}_{{c}_{k}{}^{\frown} {j} }$ for all ${k}\in\omega$ and ${n}\in\omega \cup \{-1\}$;
\item[d2.]
 $\shoot{{p}^{\alpha}\uph{n}}{\mathbf{W}^{\alpha}}\mathbin{\not\to} \{{x}^\alpha\} \cup \bigcup_{{k}\in\omega}{H}^{n}_{k}$ for all ${n}\in\omega$;
\item[d3.]
 ${E}^{n}_{k}$ is infinite for all ${k}\in\omega$ and ${n}\in\omega \cup \{-1\}$;
\item[d4.]
 ${H}^{n}_{k} \neq \varnothing$ for all ${k}\in\omega$ and ${n}\in\omega \cup \{-1\}$;
\item[d5.]
 ${H}^{n}_{k} \subseteq {H}^{{n}-1}_{k}$ for all ${k}\in\omega$ and ${n}\in\omega$;
\item[d6.]
 ${H}^{-1}_{k} \subseteq\: {S}_{{c}_{k}}\! \setminus {S}_{{c}_{{k}+1}}$ for all ${k}\in\omega$;
\item[d7.]
 ${H}^{-1}_{k}\cap{S}_{{c}_{j}} = \varnothing$ for all ${k},{j}\in\omega$ such that ${k}<{j}$;
\item[d8.]
 the family $\ll {H}^{n}_{k} \rr_{{k}\in\omega}$ is disjoint  for all ${n}\in\omega\cup \{-1\}$.
\end{itemize}
Using these families, we can define ${U}^\alpha$ as follows:
$$
{U}^\alpha \coloneq \bigcup_{{k}\in\omega} {H}^{\lh({d}_{k+1})}_{k}.
$$

Let us prove (a2).
Since ${x}^\alpha = \bigcup_{{k}\in\omega} {c}_{k}$, it follows that $\{{S}_{{c}_{k}} : {k}\in\omega\}$ is {a} neighbourhood base at ${x}^\alpha$ in the Baire space.
Using (d6), (d5), and (d4) we have
$$
 {S}_{{c}_{k}} \supseteq  {H}^{-1}_{k}
 \supseteq {H}^{\lh({d}_{k+1})}_{k} \neq \varnothing \text{ for all } {k}\in\omega,
$$
and so
$$\forall U \in \taun({x}^\alpha)\ \exists n \in \omega\ \forall k > n\ [{H}^{\lh({d}_{k+1})}_{k} \subseteq U].
$$
Now from (d1) it follows that ${H}^{\lh({d}_{k+1})}_{k}$ is closed in $\bspace$ for all $k\in\omega$, and so $\cl{{U}^{\alpha}}{\taun} = {U}^{\alpha} \cup \{{x}^\alpha\}$.

Now we  prove (a4). Let ${n}\in\omega$; we need to show that
$$
\shoot{{p}^{\alpha}\uph{n}}{\mathbf{W}^{\alpha}}\ninn\{{x}^{\alpha}\}\cup{U}^{\alpha}.
$$
Recall that
${d}_{0}={b}_{\dot{t}\uph{0}}=
{b}_{\ll\rr}={\ll\rr}$. It follows from (c3) that there is ${j}\in\omega$ such that
\begin{equation}\label{{j}} \lh({d}_{{j}})\leqslant{n}<\lh({d}_{{j}+1}).
\end{equation}
Let
$$
{A}\coloneq\bigcup_{{k} < {j}}{H}^{-1}_{k}
\quad\text{and}\quad
{B}\coloneq\{{x}^\alpha\} \cup \bigcup_{{k}\geqslant {j}}{H}^{\lh({d}_{{k}+1})}_{k}.
$$
By (d5) we have
$\{{x}^\alpha\} \cup {U}^\alpha\subseteq{A}\cup{B}$,
so using  Lemma \ref{shoot2}(1) it is enough to show that
$\shoot{{p}^\alpha\uph{n}}{\mathbf{W}^\alpha}\not\to{A}\cup{B}$.

Suppose on the contrary that $\shoot{{p}^\alpha \uph {n}}{\mathbf{W}^\alpha} \to{A}\cup{B}$.
By regularity of the scheme ${W}^{\alpha}$, (\ref{{j}}), the definition of ${p}_
\alpha$, and (c1) we have
$$
\shoot{{p}^\alpha\uph{n}}{\mathbf{W}^\alpha}\mathbin{\to}
{W}^{\alpha}_{{p}^{\alpha}\uph{n}}
\,\subseteq\,
{W}^{\alpha}_{{p}^{\alpha}\uph\lh({d}_{{j}})}
\,=\,
{W}^{\alpha}_{{d}_{{j}}}
\,\subseteq\,
{S}_{{c}_{{j}}}.
$$
Then
$\shoot{{p}^\alpha\uph{n}}{\mathbf{W}^\alpha}\to{S}_{{c}_{{j}}}$, so using Lemma \ref{shoot2}(2) and (d7) we have
$$
\shoot{{p}^\alpha\uph {n}}{\mathbf{W}^\alpha}\to{S}_{{c}_{{j}}}\cap({A}\cup{B})
\,=\,
({S}_{{c}_{{j}}}\cap{A})\cup({S}_{{c}_{{j}}}\cap{B})
\,=\,
$$
$$
\,=\,
\varnothing\cup({S}_{{c}_{{j}}}\cap{B})
\,=\,
{S}_{{c}_{{j}}}\cap{B},
$$
so 
$\shoot{{p}^\alpha \uph {n}}{\mathbf{W}^\alpha}\to{B}$ by Lemma \ref{shoot2}(2).
It follows from (d5), (c3), and (\ref{{j}}) that
$$
{B}\,=\,
\{{x}^\alpha\} \cup \bigcup_{{k}\geqslant {j}}{H}^{\lh({d}_{{k}+1})}_{k}
\,\subseteq\,
\{{x}^\alpha\} \cup \bigcup_{{k} \geqslant {j}}{H}^{\lh({d}_{{j}+1})}_{k}
\,\subseteq\,
$$
$$
\,\subseteq\,\{{x}^\alpha\} \cup \bigcup_{{k} \geqslant {j}}{H}^{n}_{k}
\,\subseteq\,
\{{x}^\alpha\} \cup \bigcup_{{k}\in\omega}{H}^{n}_{k},
$$
therefore 
$\shoot{{p}^\alpha \uph {n}}{\mathbf{W}^\alpha} \mathbin{\to}\{{x}^\alpha\} \cup \bigcup_{{k}\in\omega}{H}^{n}_{k}$, which contradicts (d2).

It remains to build the indexed families
$$\ll {E}^{n}_{k} : {k}\in\omega, {n}\in\omega \cup \{-1\}\rr
\quad\text{and}\quad
\ll {H}^{n}_{k} : {k}\in\omega, {n}\in\omega \cup \{-1\}\rr.$$
We do it by recursion on ${n}\in\omega\cup\{{-}1\}$.
Recall that ${x}^\alpha\colon\omega\to\omega$ and note that 
$$
{c}_{k}=\big\ll{x}^\alpha({0}),\ldots,{x}^\alpha\big(\lh({c}_{k})-1\big)\big\rr.
$$ When ${n}=-1$, for all ${k}\in\omega$, we put
$$
{E}^{-1}_{k} \coloneq \big\{{j}\in\omega : j > {x}^\alpha\big(\lh({c}_{k})\big)\big\}
\quad\text{and}\quad
{H}^{-1}_{k} \coloneq \bigcup_{{j}\in{E}^{-1}_{k}} {S}_{{c}_{k}{}^{\frown}{j}}.
$$
Note that these sets satisfy properties (d1)--(d8) for ${n}=-1$; in particular, (d6)--(d8) follow from (c2).

Now suppose that the sets ${E}^{l}_{k}$ and ${H}^{l}_{k}$ have been chosen for all ${l} < {n}$ and ${k}\in\omega$, and that they satisfy (d1)--(d8).
Using (d4), for all ${k}\in\omega$, fix a disjoint pair of infinite sets ${E}^{{n}-1}_{k}(0)$ and ${E}^{{n}-1}_{k}(1)$ such that
$$
{E}^{{n}-1}_{k} = {E}^{{n}-1}_{k}(0) \cup {E}^{{n}-1}_{k}(1).
$$
Now, for ${m}\in\{0, 1\}$, set
$$
{H}^{{n}-1}_{k}({m}) \coloneq \bigcup_{{j}\in{E}^{{n}-1}_{k}({m})} {S}_{{c}_{k}{}^{\frown} {j} }
\quad\text{and}\quad
{G}({m}) \coloneq \bigcup_{{k}\in\omega}{H}^{{n}-1}_{k}({m}).
$$
We have ${H}^{{n}-1}_{k}={H}^{{n}-1}_{k}(0)\cup{H}^{{n}-1}_{k}(1)$ and
${H}^{{n}-1}_{k}(0)\cap{H}^{{n}-1}_{k}(1)=\varnothing$
for all ${k}\in\omega$, so it follows from (d8) that
$$
 {G}(0) \cap {G}(1) = \varnothing.
$$

The scheme ${W}^\alpha$ is open on the Baire space and has nonempty leaves, so $\shoot{{p}^\alpha\uph{n}}{\mathbf{W}^\alpha}\not\to \{{{x}^\alpha}\}$.
We have $\{{{x}^\alpha}\}=\big(\{{{x}^\alpha}\}\cup{G}(0)\big)\cap\big(\{{{x}^\alpha}\}\cup{G}(1)\big)$,
so it follows from Lemma~\ref{shoot2}(2) that
$$
 \shoot{{p}^\alpha \uph {n}}{\mathbf{W}^\alpha} \not\to \{{{x}^\alpha}\}\cup{G}(\dot{m})
$$
for some $\dot{m}\in\{0, 1\}$.
Put
$$
{E}^{n}_{k}\coloneq{E}^{{n}-1}_{k}(\dot{m})
\quad\text{and}\quad
{H}^{n}_{k}\coloneq{H}^{{n}-1}_{k}(\dot{m}).
$$
It is straightforward to check that conditions (d1)--(d8) are satisfied.
\end{proof}

\begin{proof}[\textbf{Proof of Theorem~\ref{teo.main}}]
 We must find a zero-dimensional $\pi$-space that is not a continuous open image of a space with a Lusin $\pi$-base.
 
 Let ${F}$ be the function from Proposition \ref{function.F}. It follows from Proposition \ref{stand_pi_sp} that  there exists a zero-dimensional standard $\pi$-space $\ll \bset, \tau \rr$ such that ${F}({x}) \cup \{{x}\} \in \tau$ for all ${x} \in \bset$.

 We will show that $\ll \bset, \tau \rr$ has no $\alpha$-scheme. Then $\ll \bset, \tau \rr$ is not a continuous open image of a space with a Lusin $\pi$-base by Proposition \ref{desc_open_LPB}.

 Assume by contradiction that ${\bf {V}}$ is an $\alpha$-scheme for $\ll \bset, \tau \rr$. By Lemma \ref{scheme_native}, ${\bf {V}}$ is a semi-open Souslin scheme on $\bspace$ and by Proposition \ref{LP_pi_base}, ${\bf {V}}$ is a $\pi$-base Souslin scheme on $\ll \bset, \tau \rr$.

 Since $\ll \bset, \tau \rr$ is a standard $\pi$-space and ${\bf {V}}$ is a regular $\pi$-base open Souslin scheme with nonempty leaves on $\ll \bset, \tau \rr$ it follows that ${\bf {W}} \coloneq \mathsf{Int}({\bf {V}},\taun)$ is a regular open $\pi$-base Souslin scheme on $\bspace$ and ${\bf {W}}$ has nonempty leaves. By the choice of ${F}$, there exist ${p}, {x} \in \bset$ such that
 \begin{itemize}
  \item[i.] $\fruit{{p}}{{\bf {W}}} = \{{x}\}$;
  \item[ii.] $\shoot{{p}\uph{n}}{{\bf {W}}}\ninn\{{x}\}\cup{F}({x})$ for all ${n}\in\omega$;
  \item[iii.] $\fruit{{p}}{\mathbf{{V}}} = \{{x}\}$.
 \end{itemize}

 Let us show that ${\bf {V}}$ does not satisfy condition (S2) of the definition of $\alpha\nos$-scheme. 
 Using (iii), it is enough to show that the formula ${p}\xrightarrow{\mathbf{V}\!,\tau} {x}$ is not true.
 From (ii) it follows that
 \[
  \shoot{{p}\uph{n}}{{\mathbf{{V}}}}\ninn\{{x}\}\cup{F}({x}) \text{ for all } {n} \in\omega.
 \]
 It remains to note that ${F}({x}) \cup \{{x}\} \in \tau({x}).$
\end{proof}

\begin{ques}
Does there exist a Hausdorff compact space that is a continuous open image of a $\pi$-space but is not a continuous open image of a space with a Lusin $\pi$-base?
\end{ques}

\bibliographystyle{plain}
\bibliography{books.bib}
\end{document}